\documentclass[11pt]{article}

\usepackage{amsmath,amsfonts,amsthm,amscd,amssymb,latexsym}
\usepackage{enumerate,graphicx,psfrag,subfigure}

\usepackage[usenames]{color}
\usepackage{soul}

\newtheorem{theorem}{Theorem}[section]

\newtheorem{lemma}[theorem]{Lemma}
\theoremstyle{definition}

\newtheorem{corollary}[theorem]{Corollary}

\newtheorem{example}[theorem]{Example}

\newtheorem{remark}[theorem]{Remark}
\newtheorem{proposition}[theorem]{Proposition}

\def\A{{\mathcal{A}}}

\def\S{\Sigma}
\def\Ss{\Sigma^{\ast}}

\def\d{{\delta}}

\newcommand{\fracd}[2]{{\displaystyle\frac{#1}{#2}}}

\newcommand{\bi}{\begin{itemize}}
\newcommand{\ei}{\end{itemize}}
\newcommand{\ben}{\begin{enumerate}}
\newcommand{\een}{\end{enumerate}}

\author{\textsf{Elizaveta Frenkel},  \textsf{Vladimir N. Remeslennikov}\footnote{The second author was supported by following grants: results of section \ref{sec:comp} and Theorem \ref{th:dbcone} in particular was supported by RNF grant 14-11-00085 and the rest of the paper was written with a support of RFFI grant 14-01-00068.}}
\title{Measuring cones and other thick subsets in free groups}

\begin{document}

\maketitle


\begin{abstract} In this paper we investigate the special automata over finite rank free groups and estimate asymptotic characteristics of sets they accept. We show how one can decompose an arbitrary regular subset of a finite rank free group into disjoint union of sets accepted by special automata or special monoids. These automata allow us to compute explicitly generating functions, $\lambda-$measures and Cesaro measure of thick monoids. Also we improve the asymptotic classification of regular subsets in free groups.
\end{abstract}



\section{Introduction}

This paper continue the series of papers written by different authors \cite{multiplicative,af_semr,fmrI,fmrII,dc}. More specifically, we expand the results of \cite{almaz} and give their proves. We return to the question of asymptotic classification of regular subsets in finite rank free groups, thus being motivated by needs of universal algebraic geometry. Namely, having in mind the notion of an $A-$dimension function over arbitrary algebraic structure introduced by the second author and its applications in different algebraic systems (see \cite{DMR}), we have started to prepare the algebraic and algorithmic foundations for a suitable dimension function in group theory. In particular, in a sequel paper we are going to present such an algorithm for regular subsets of finite rank free groups over certain group $A$.
However, the existing asymptotic classification of sets, appeared first in \cite{multiplicative} and then refined in \cite{af_semr} does not allow us to fulfill this task. To reveal the problem, we formulate these results (see Section \ref{sec:notations} for definitions):

{\bf Theorem. \label{th:regular-negl-thick} }{\it  Let $F$ be a finite rank free group. Then
\begin{itemize} 
\item [1)] every regular subset of $F$ is either thick or exponentially negligible;
\item [2)] a regular subset of $F$ is thick if and only if its prefix closure contains a cone.
\end{itemize} }

As we shall see below, all the necessary computations can be easily done in the case of regular exponentially negligible sets. The missing bit consists in more specific characterisation of thick sets and in finding the way to distinguish between them in a finer way. The present paper covers these problems.
The following theorem adds to our knowledge on how does the thick sets look like; these new details turn out to be crucial as we shall show in section \ref{sec:comp}:

{\bf Theorem \ref{thick}. }{\it   A regular subset $R$ of $F$ is thick if and only if it contains a subset $w\circ T$, with $T$ being a thick monoid and $w \in F$.}

This theorem was formulated in \cite{almaz}, Theorem 8, in slightly different form; this fact is related to another understanding of the word ``contain''. 
Another important results of the current paper concerns new algorithms for the computation of the generating functions and Cesaro measure of sets recognised by so-called special automata and thick monoids. The algorithms we suggest appears to be easier with respect to the older ones.

Now, a few words on the structure of the paper.

In Section \ref{sec:notations} we give some basics on regular sets and recall techniques  for measuring subsets
in  a free group $F$ and the asymptotic classification of regular sets. In subsection \ref{sub:gen_alg1} we also provide Algorithm I for computation of the generating function of a regular set by means of linear algebra.

Section \ref{sec:specials} starts with the definition of a special automaton over monoid and group. Further on we prove that every regular subset $L$ of a finite rank free group $F$ can be represented as a finite disjoint union of languages accepted by certain type of automata (see Proposition \ref{Decomp1}), which is going to be crucial property of the special automata in a context of both current paper and the construction of a dimension function in free groups. In Lemma \ref{lenondegen} we also show how one can split the sets accepted by special automata, which leads to the notion of a thick monoid.
Further in this section we  analyse the structure and compute the most important asymptotic characteristics of thick monoids: the generating function and the Cesaro measure among the most important of them (see Proposition \ref{notxcomplete} and Lemma \ref{le:thmonlang}). In Theorem \ref{thick} we improve already mentioned result on the asymptotic classification of regular sets. We conclude Section \ref{sec:specials} with Algorithm II computing $\lambda-$measure of a negligible set accepted by a special automaton.

The next Section \ref{sec:comp} is dedicated to computations and also it reassumes the results from above. Preliminary calculations made in Lemma \ref{nonio} allows us to compute the generating function of an arbitrary double-based cone (see Theorem \ref{th:dbcone}). We want to emphasize that this is a crucial theorem for all the paper, interesting {\emph{ per se}}, applied in Lemma \ref{le:thmonlang}, and having a lot of structural and computational consequences. In particular, Theorem \ref{th:dbcone} reduces Algorithm I to much more straightforward combinatorial calculations and formulae that does not use linear algebra methods.


\section{Regular sets in free groups}\label{sec:notations}
In this section we recall the main definitions and tools of particular interest for our purposes.
\subsection{Regular sets: some properties}\label{sub:reg} We assume that the reader is familiar with basic facts on regular sets in monoids and groups (described in details, for example, in \cite{eps,lawson}). Let $X = \{ x_1, \ldots, x_m \}$ be an alphabet and define $\S$ to be the letters of $X$ with their formal inverses: $\S = X \cup X^{-1}$. Let $F=F(X)$ be the free group generated by $X$.
A {\em finite state automaton} $\A$ is a quintuple $(S, \S,\d,I,Z)$, where $S$ is a finite set of states, $\S$ is an alphabet, $I \subset S$ is the (non-empty) set of initial states, $Z \subseteq S$ is the set of final states, and $\d$ is a set of arrows with labels in the enlarged alphabet $\S \cup \varepsilon$ (here $\varepsilon$ is assumed not to lie in $\S$). Further, a {\em deterministic automaton} can be considered a special case of a finite state automaton, with no arrows labelled $\varepsilon$, the only one initial state and each state being the source of exactly one arrow with any given label from $\S$.
By the Kleene-Rabin-Scott theorem, all regular subsets over $\S$ (i.e. the closure of finite subsets of free monoid over $\S$ under the rational operations) are exactly the sets accepted by a finite state automaton over $\S \cup \epsilon$, or, equivalently, accepted by a deterministic automaton over $\S$. The language accepted by an automaton $\A$ we shall denote by $L=L(\A)$.

\subsection{Multiplicative measures: basics and first algorithms}\label{sub:gen_alg1}
We  denote by $|f|$ the length of an element $f \in F$, and let $S_k = \{\, w \in F \mid |w|=k\,\}$ denote the sphere of radius $k$ in $F$. We consider a subset $R$ of $F$,  and denote by $f_k(R)= \frac{\vert R \cap S_k\vert}{\vert S_k\vert}$ the {\em  frequency} of elements from $R$ among the words of length $k$ in $F$.

\paragraph{$\lambda-$ measure.} An important measuring tool in $F$ is the so-called frequency measure, introduced in \cite{multiplicative} and studied in \cite{fmrI} and \cite{fmrII}. By definition, $$ \lambda(R) = \sum_{k=0}^\infty f_k(R).$$
A subset $R \subseteq F$ is called {\em $\lambda$-measurable}, if $\lambda(R) < \infty$, and {\em exponentially $\lambda-$measurable} if there exists a positive constant $\delta <1$ such that $f_k(R) < \delta^k$ for big enough $k$. We adjuste this measure to obtain $\lambda^*(R) = \frac{2m}{2m-1}\lambda(R)$.

\paragraph{Generating function.} One can consider the {\em (frequency) generating function} for $R$ as a formal series
in $R[[t]]$: $g_R(t) = \sum_{k=0}^\infty f_k(R) t^k$. We shall also use the {\em adjusted} version of this function: $g^{\ast}_R(t) =\frac{2m}{2m-1} \cdot g_R(t)$.
In case of regular subsets of $F$ the generating function can be described in a very concise form:

\begin{theorem}\label{observ_g} For a regular set $R \subseteq F$ the function $g_R(t)$ is a rational function of $t$ with rational coefficients and either
\begin{itemize}

\item has no singularity at $t=1$ (in this case $R$ is exponentially $\lambda-$measurable) or

\item has a simple pole at $t=1$ (in this case $R$ is thick\footnote{The rationality of $g_R(t)$ for regular sets is well known (for instance, in \cite{flajolet}; it follows also from Algorithm I for $g_R(t)$ below), while statements about asymptotic properties of $R$ follow from asymptotic classification of regular sets shown in \cite{multiplicative,af_semr}.}).
\end{itemize}

In particular,
\begin{equation}\label{mu0} Res_{1}g_R(t)= - \mu_0(R).
\end{equation}
\end{theorem}
Recall that a regular set is called {\em thick} if the parameter $\mu_0(R)$ defined by formula (\ref{mu0}) is strictly positive. This parameter $\mu_0(R)$ is called {\em Cesaro density} of $R$.
We use often the following simple properties of the generating function: suppose $R_1$ and $R_2$ are regular subsets of $F$. Then
\begin{enumerate}\label{prop}
\item [(1.)] \label{exgen0} If $R=R_1 \cup R_2$, then the corresponding generating function can be computed as $g_R(t)=g_{R_1}(t)+ g_{R_2}(t) - g_{R_1\cap R_2}(t)$.
\item [(2.)] \label{exgen1} If $R=R_1\circ R_2$, then $g_R(t)=g_{R_1}(t)g^{\ast}_{R_2}(t)$.
\end{enumerate}

Now we describe the first algorithm for calculation of the generating function for an arbitrary regular subset of a finite rank free group $F$. This algorithm is previously known (see, for example, \cite{multiplicative}), although it was not directly formulated there.

{\bf{ Algorithm I computing the frequency generating function $g_R(t)$ for a regular set $R$.}}
Indeed, let $\mathcal{A} = (S, \S,\d,I,Z) $ be an automaton such that $|S| = n$ and let $A$ be  it's {\textrm{adjacency matrix}}, i.e. $n \times n$  matrix with entries $a_{ij}$ such that each $a_{ij}$ corresponds to the number of arrows from the state $i$ to the state $j$. Clearly, the number of different paths of length $k$ from $i$ to $j$ is equal to $(A^k)_{i,j}$. Denote by $R$ the subset of $F$ accepted by $\mathcal{A}$.
{\bf{Algorithm I:}}
\begin{itemize}
\item[1.] Given an automaton $\mathcal{A}$, compute the entries $a_{ij}$, $i,j = 1, \ldots, n$ of the adjacency matrix $A$.
\item[2.] Compute the entries $b_{ij}$ of the fundamental matrix $B =  tA (E-tA)^{-1}$ of $\mathcal{A}$, with the entries $b_{ij}$ from the ring of formal power series $R[[t]]$.
\item[3.] The generating function $g_R(t)$ is equal to $\sum\limits_{i \in I, j \in Z} b_{ij}$.
\end{itemize}

One of the disadvantages of the Algorithm I  is that step [2.] involves the matrix inversion, and it makes the algorithm hardly implementable with the size of automaton $n$ big enough.
However, computation of the generating function can be significantly simplified for a wide class of regular sets. In what follows we introduce this type of sets and describe their structure along with the improved algorithm for computation of $g(t)$.
Now, using Algorithm I and the properties above, we compute generating function for certain regular subsets of $F$. We also calculate the corresponding values of Cesaro density $\mu_0(R)$ (defined by formula (\ref{mu0})).

\begin{example}\label{ex} 
\begin{enumerate}
\item \label{exgen2} For a whole free group $F$ we have $g_F(t)= - \frac{1}{t-1}$ and $\mu_0(F) = 1$.
\item \label{exgen3} For a set $R= F^{\sharp} = F \setminus\{1\}$ we have $g_{R}(t)=\frac{- t}{t-1}$ while $\mu_0(F^{\sharp}) =  1$.
\item\label{exgen4} Let $R$ be a cone\footnote{We recall the notions of cones in \ref{sub:thickthin}}  $C(w)$ or $R=C[w]$, and let $|w| = r$. Then
$g_{R}(t)= \frac{1}{2m (2m-1)^{r-1}} \cdot \frac{-t^r}{t-1},$ and $$\mu_0(R) = \frac{1}{2m (2m-1)^{r-1}}.$$
\item\label{exgen5} If $R=F\setminus B_{r-1}$, then $g_R(t)= \frac{-t^r}{t-1},$ and $\mu_0(R) = 1$.
\item\label{exgen6} For a subgroup $H < F(X)$ of all words of even length direct calculations of frequency generating functions gives $g_H(t)=\frac{1}{1-t^2},$ and therefore $\mu_0(H) = \fracd{1}{2}$.
\end{enumerate}
\end{example}
%
%
%

\section{Special automata over free groups and monoids}\label{sec:specials}

In this section we investigate one of the central concepts of this paper, i.e. special automata over monoids and groups. We show in Proposition \ref{Decomp1} that every regular set in a free group can be decomposed into finite union of subsets accepted by special automata.

\subsection{Definitions}\label{sub:def}
Let $\A = (S, \S, \d, i_0, Z)$ be a deterministic automaton. $\A$ is called {\em special} over the monoid $\Ss$ if
\begin{itemize}
\item [a.]\label{sa} The initial vertex has no inedges;
\item [b.]\label{sb} There is only one final state $z_0 \in Z$;
\item [c.]\label{sc} $\A$ does not contain inaccessible states;
\item [d.]\label{sd} For every state $s \in S$  there is a direct path from $s$ to the final state $z_0$;
\item [e.]\label{se} For any state $s \in S$, all arrows which enter $s$ have the same label $x \in \S$ (we shall say, $s$ {\em has type $x$}).
\end{itemize}
In order to adjust the notion of speciality to groups, we impose an additional constraint on automata. Namely, let $F$ be the free group, and $\A$ be a special automaton. Suppose also that
 \begin{enumerate}
 \item[f.] \label{sf} For any state $s$ of type $x$ in $\A$, all arrows exiting from $s$ cannot have label $x^{-1}$.
\end{enumerate}
  $\A$ is a {\em special} automaton over the group $F$, if it satisfies the conditions (a)--(f).

In what follows we also shall use a notion of a special monoid. Namely, a monoid $M$ is called {\em special} if it is accepted by a finite automata $\A$ with $i_0=z_0$, satisfying conditions (b) -- (f).

\subsection{Decomposition into special automata}\label{sub:decomp}
\begin{proposition}\label{Decomp1} Let $L$ be a regular language in $F$. Then there exist a finite number of automata $\A_0, \ldots, \A_k$ such that
\begin{itemize}
\item $L$ is a disjoint union of languages $L_0= L(\A_0), \ldots, L_k = L(\A_k)$ in $F$: $L=L_0 \sqcup L_1 \sqcup \ldots \sqcup L_k$;
\item every $L_i$ is either accepted by a special automaton or a special monoid.
\end{itemize}
\end{proposition}
\begin{proof} Since $L$ is regular in $F$, it is accepted by a finite automaton $\A$. Although we can assume that $\A$ is deterministic automaton, it will be more convenient for us to start with a non-deterministic one, which accepts $L$ as a language of reduced words, satisfies (c) (it is always possible, see \cite{Gil}), but, probably, has $\epsilon-$transitions and more than one initial state.
Therefore, $\A$ has a form $\A = (S, \S\cup \epsilon, \d, I, Z)$. We begin with an application of Rabin-Scott powerset construction (see \cite{RS} for details).
As an output of this procedure, we obtain an automaton $\A' = (S', \S, \d', i_0, Z')$, which does not have $\epsilon$-transitions and has only one initial state $i_0$ without inedges, as required. As a by-product of the construction, we have conditions (c) and (d) satisfied. Further, because we have started from the automaton $\A$ which does not have consecutive $x, x^{-1} (x \in \S)$ transitions, $\A'$ does not have these transitions as well. Nevertheless, it might happen that $\A'$ has more than one final state and some states of $S'$ have incoming edges with different labels.
In the latter case we split the states of $\A'$ as it is shown on figure \ref{RIS.Spl}:
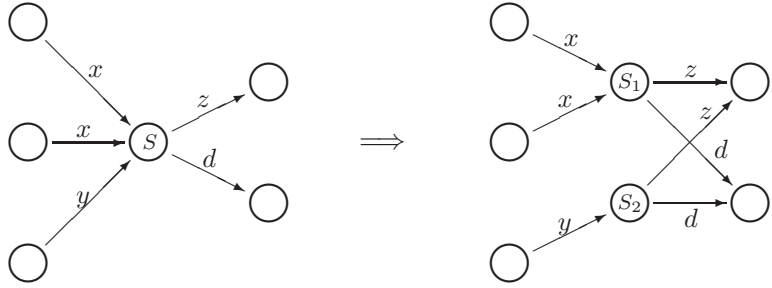
\begin{figure}[h]
 \setlength{\unitlength}{0.8mm}
\begin{picture}(100,70)(-15,-10)

\thicklines \put(55,19){$\Longrightarrow$}

\put(0,0){\circle{6}} \put(0,20){\circle{6}}
\put(0,40){\circle{6}} \put(20,20){\circle{6}}
\put(40,10){\circle{6}} \put(40,30){\circle{6}}

\put(80,0){\circle{6}} \put(80,20){\circle{6}}
\put(80,40){\circle{6}} \put(100,10){\circle{6}}
\put(100,30){\circle{6}} \put(120,10){\circle{6}}
\put(120,30){\circle{6}}

\thinlines \put(3,3){\vector(1,1){14}}
\put(4,20){\vector(1,0){12}} \put(3,37){\vector(1,-1){14}}
\put(24,18){\vector(2,-1){12}} \put(24,22){\vector(2,1){12}}

\put(84,2){\vector(2,1){12}}
\put(84,22){\vector(2,1){12}}
\put(84,38){\vector(2,-1){12}}

\put(104,10){\vector(1,0){12}}
\put(103,13){\vector(1,1){14}}
\put(104,30){\vector(1,0){12}}
\put(103,27){\vector(1,-1){14}}

\put(19,19){\scriptsize {$S$}} \put(98,9){\scriptsize {$S_2$}}
\put(98,29){\scriptsize {$S_1$}}

\put(8,10){\small {$y$}} \put(8,21){\small {$x$}}
\put(10,31){\small {$x$}} \put(29,16){\small {$d$}}
\put(28,26){\small {$z$}}

\put(88,5.5){\small {$y$}} \put(88,26){\small {$x$}}
\put(89,36){\small {$x$}} \put(109,6){\small {$d$}}
\put(114,17.5){\small {$d$}} \put(111.5,24){\small {$z$}}
\put(109,31){\small {$z$}}

\end{picture}

\caption{{\small \sl Splitting the states of the automaton $\A'$.}} \label{RIS.Spl}
\end{figure}

The output of the splitting procedure we shall call $\A''=(S'', \S, \d'', i_0, Z'')$. If $\A''$ has only one final state $z_0 \neq i_0$, then it is  special over $F$. If $Z'' = \{z_0 \}$ and $i_0 = z_0$, then $L(\A'')$ is a special monoid by definition and due to (f).
Suppose now $Z'' = \{ z_0, \ldots, z_k \}$, with $k\ge 1$. For every $z_i \in Z''$ consider the maximal connected subgraph $\A_i = (S_i, \S, \d_i, i_0, z_i)$ of $\A''$ such that $S_i \subset S''$ and $\d_i \subset \d''$ induced by the paths of arrows from $i_0$ to $z_i$; obviously, there are two options for $\A_i$: either $\A_i$ has distinct initial and finale state and therefore $A_i$ is special, or initial and final states conincide and so $L(\A_i)$ is a monoid . Since $L= L(\A'')$ and $\A''$ satisfies (c), (d), clearly, $L = L(\A_0)\cup L(\A_2)\cup\ldots\cup L(\A_k)$. Moreover, this union is disjoint since $L_i \cap L_j \neq \emptyset$ implies existence of paths of arrows $p_1, p_2$ such that $p_1= i_0 v_1 \ldots v_s z_i \in \A''$ and $p_2 = i_0 u_1 \ldots u_r z_j \in \A''$ for $z_i \neq z_j$, with the label $\d''(p_1) = \d''(p_2)$, a contradiction with $\A''$ being deterministic.\end{proof}

\begin{remark}\label{remark} Notice that the number $k$ and subsets $L_i$ for different decompositions of $L$ can vary. On  the other hand, suppose $L = L_0 \sqcup L_2 \sqcup \ldots \sqcup L_k$ and $L= M_0 \sqcup M_2 \sqcup \ldots \sqcup M_s$ are different decomposition of a regular set $L$ as in Proposition \ref{Decomp1}. Then by property (1.) of the generating functions we have $g_{L_0}(t)+\ldots + g_{L_k}(t) = g_L(t) = g_{M_0}(t)+\ldots + g_{M_s}(t)$.
\end{remark}

\subsection{Further splitting of subsets in free groups}\label{sub:splitting}
 A special automaton satisfying (a)--(f) in turn admits further splitting:
\begin{lemma}\label{lenondegen} Let $R = R(\A)$ and $\A = (S, \S, \d, i_0, z_0) $ be a special automaton over $F$. Then there exist regular languages $R_1, R_2, R_3 \subset F$ such that $R_j$ are accepted by $\A_j=(S, \S_j, \d_j, i_j, z_j)$, $\A_1$ is special over $F$ and
\begin{itemize}
\item[1.]if $\A$ has at least one arrow exiting $z_0$, then $R_2$ is non-empty and $i_2 = z_2$, while $i_3 \neq z_3$ and
\begin{equation}\label{RR1R2}
R = R_1 \circ R_2 {\textrm{ is unambiguos}};
\end{equation}
\begin{equation}\label{R2R3}
R_2 = 1 \sqcup R_3 \sqcup (R_3 \circ R_3) \sqcup (R_3 \circ R_3 \circ R_3) \sqcup \cdots;
\end{equation}
\begin{equation}\label{gen}
g_R(t) = g_{R_1}(t)g^*_{R_2}(t); \,\,\,\, \lambda(R) = \lambda(R_1)\lambda^*(R_2).
\end{equation}
\item[2.] if there is no arrows exiting $z_0$, then $R_2 = R_3 = \emptyset$, $R=R_1$, $\lambda(R) = \lambda(R_1)$, and $g_R(t) = g_{R_1}(t)$.
\end{itemize}
\end{lemma}
\begin{proof} Although the construction of sets $R_1, R_2, R_3$ and their $\lambda-$measures appears in \cite{multiplicative} and \cite{fmrII}, we shall widely use these sets and automata in what follows, and therefore we repeat briefly the necessary computations (see also Example \ref{exnondeg} and its illustrations in figures \ref{fignondeg}, \ref{findeg1}, \ref{figdeg2}, \ref{figdeg3}).

Suppose that the final state of $\A$ does not have exiting arrows. Then we leave $\A$ as it is, and, clearly, [2.] holds.

Let now $z_0$ has at least one exiting arrow. In this case the special automaton $\A_1$ accepting $R_1$ can be obtained from $\A$ by removing all arrows exiting from $z_0$; we take $i_1 = i_0$ and $z_1  = z_0$. Let us consider the automaton $\A_2$ accepting $R_2 \neq \emptyset$  formed by all states accessible from the state $z_0$, with the same arrows between them as in $\A$; we take $z_0$ for the both $i_2$ and $z_2$. If now $u \in R_1$ and $v \in R_2$, then the word $uv$ is reduced and $\lambda(uv) = \lambda(u)\lambda^*(v).$ Therefore, the presentation of $R$ in the form $R = R_1 \circ R_2$ is unambiguous. Indeed, let $w \in R$ can be written in two different forms as $u_1 \circ v_1$ and $u_2\circ v_2$, where $u_1, u_2 \in R_1$ and $v_1,v_2 \in R_2$. Assume that $|u_1| > |u_2|$ (otherwise consider the pair $v_1$ and $v_2$), and let $h  = u_2^{-1} u_1 \in F$ be readable in $\A$. Notice that $h$ starts at $z_0$ since $u_2$ is accepted by $\A_1$ and ends at $z_0$ because $\A_1$ accepts $u_1$. Therefore, $h$ is accepted by $\A_2$, a contradiction with the construction of $\A_1$. The estimates on $\lambda(R)$ and $g_R(t)$ now follow immediately from the construction (frequencies assigned to arrows in $\A_1$, $\A_2$ the
same as they were in $\A$) and formula (2.).

Further, we transform the automaton $\A_2$ by splitting the final state $z_2 = z_0$ into separate initial state $i_3$ (with no arrows entering it, and those arrows which were exiting $z_2$ now exiting $i_3$), and the final state $z_3$ (with no arrows exiting
$z_3$, and those arrows which were entering $z_0$ now entering $z_3$). Then, clearly, (\ref{R2R3}) holds.
%

\end{proof}
\begin{corollary}\label{comonoid}  Let $R = R(\A)$ and $\A  = (S, \S, \d, i_0, z_0) $ be a special automaton over $F$, and let $R_1, R_2, R_3 \subset F$ be regular languages such that $R_1$ is accepted by a special automaton over $F$, $R_2$ is non-empty set such that its initial and the final state coincide, $R_3$ is accepted by a special automaton over $F$; $R = R_1 \circ R_2$,  and $R_2$ satisfies (\ref{R2R3}) as in lemma above. Then the subset $R_2$ of $F$ is the free special monoid generated by $\{w_i | i \in I \}$, where $w_i \in F$ are words in $R_3$ and $w_i$ can be computed effectively by $\A$.
\end{corollary}
\begin{proof} The automaton $\A_2$ constructed in the proof of Lemma \ref{lenondegen} has $i_2 = z_2$, and its final vertex $z_2$ is of $x-$type, for some $x \in \S$. The condition (f) provided by the speciality of $\A$ guarantees that the arrow labelled $x^{-1}$ cannot exit from $i_2$. Therefore, if $u_1, u_2$ are accepted by $\A_2$, then $u_1 u_2 = u_1 \circ u_2$. In particular, using the further splitting of $R_2$, one can express every $u_i$ as a reduced product of $w_i$'s accepted by (non-empty) $R_3$. Since the identity belongs to $R_2$, $\A_2$ accepts the free special monoid with generators $w_i, i \in I$.
\end{proof}

The subsets and automata described in Lemma \ref{lenondegen}, claim 1. are of particular interest for us.  Regular sets $R \subseteq F$ of such form we shall call {\em saturated}.  Sets $R_1, R_2$, and $R_3$ in the splitting defined in this lemma we shall call a {\em set of first, second}, and {\em third type}, respectively.
In what follows, we use the notations $\A_1, \A_2, \A_3$ for the splitting of arbitrary automaton $\A$ and $R_1, R_2, R_3$ for the corresponding regular sets exclusively in a sense of Lemma \ref{lenondegen}. We provide an example of such automata and sets below.

\begin{example}\label{exnondeg} Let $\S$ be an alphabet $x, y, t, z$ and the inversion is given by the rule $t \rightarrow x^{-1}, z \rightarrow y^{-1}$ (so $X = \{x, y\}$).
Consider the special automaton $\A$ (the arrow with a tale corresponds to the initial state, and the finale state is drawn as a double circle).

\begin{figure}[ht!]
\begin{center}

\psfrag{i}{$i_0$}
\psfrag{s1}{$s_1$}
\psfrag{s2}{$s_2$}
\psfrag{z}{$z_0$}
\psfrag{x}{$x$}
\psfrag{y}{$y$}
\psfrag{t}{$t$}

\includegraphics[width=6cm]{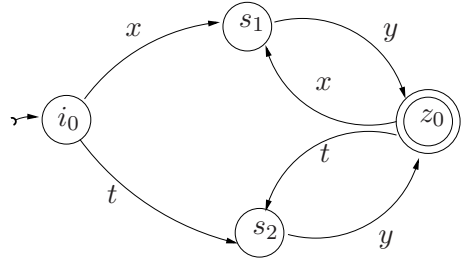}\\
\caption{The special automaton $\A$.}\label{fignondeg}
\end{center}
\end{figure}

Clearly, $R = R(\A)$ is generated by the following regular expression:
$$R = x y \left( (x^{-1} y )^{\ast} (x y )^{\ast}\right)^{\ast} \cup x y \left( (x y )^{\ast}(x^{-1} y )^{\ast} \right)^{\ast} \bigsqcup$$
$$ \bigsqcup x^{-1} y \left( (x y )^{\ast} (x^{-1} y )^{\ast}\right)^{\ast} \cup x^{-1} y \left( (x^{-1} y )^{\ast}(x y )^{\ast} \right)^{\ast}.$$

The set of first type $R_1$ can be read off by the automaton $\A_1$ shown in figure \ref{findeg1}; therefore, $R_1 = xy \cup x^{-1}y$.

\begin{figure}[ht!]
\begin{center}

\psfrag{i}{$i_0$}
\psfrag{i1}{$i_0$}
\psfrag{s1}{$s_1$}
\psfrag{s2}{$s_2$}
\psfrag{z}{$z_0$}
\psfrag{z1}{$z_0$}
\psfrag{x}{$x$}
\psfrag{y}{$y$}
\psfrag{t}{$t$}

\includegraphics[scale=.35]{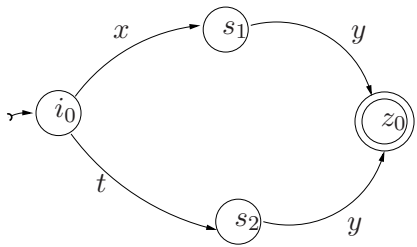}\\
\caption{The special automaton $\A_1$.}\label{findeg1}
\end{center}
\end{figure}

The sets of second type $R_2=\left((xy)^{\ast}(x^{-1}y)^{\ast} \right)^{\ast}$ and third type $R_3 = xy \cup x^{-1}y$  with their automata $\A_2$ and $\A_3$. Clearly, the elements $w_1 = xy, w_2 = x^{-1}y$ provides a set of generators for the monoid $R_2$ (see Corollary \ref{comonoid}).

\begin{figure}[ht!]
\begin{center}

\psfrag{i}{$i_0$}

\psfrag{z}{$z_0$}

\psfrag{i1}{$i_1$}

\psfrag{z1}{$z_1$}

\psfrag{s1}{$s_1$}

\psfrag{s2}{$s_2$}

\psfrag{x}{$x$}

\psfrag{y}{$y$}

\psfrag{t}{$t$}

\subfigure[$\A_2$]{
\includegraphics[scale=.35]{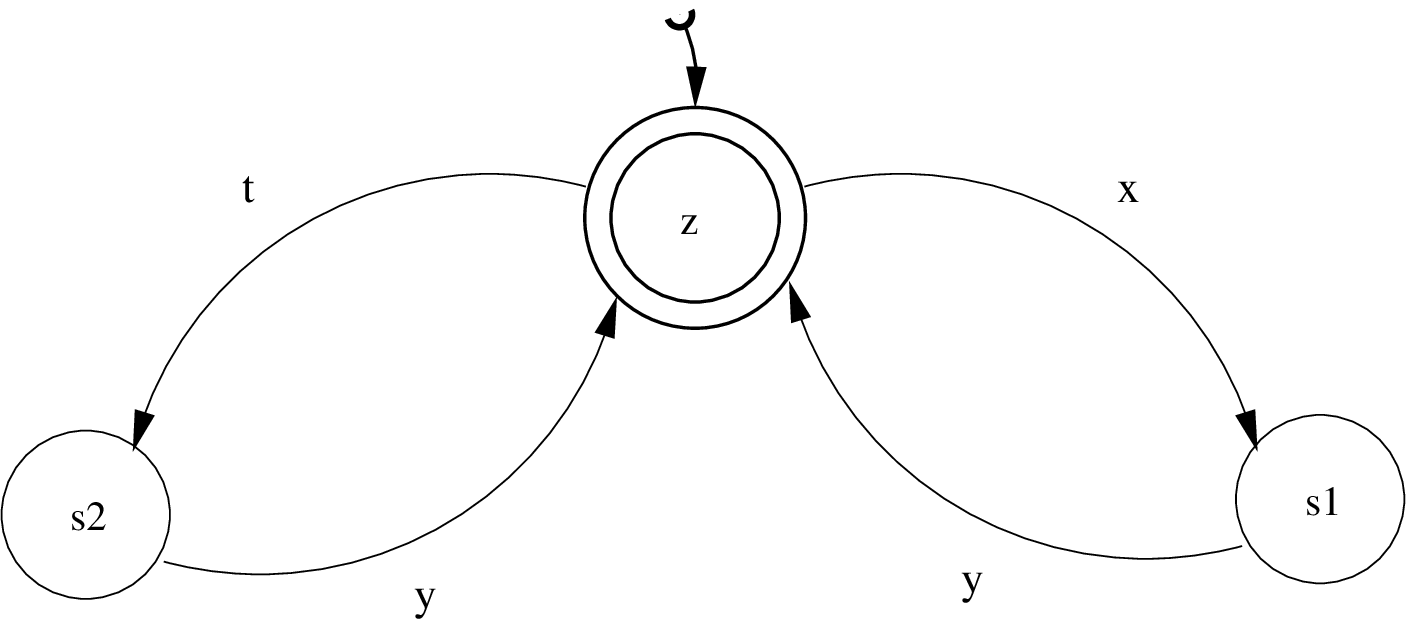}\label{figdeg2}}  \hspace{9mm}
\subfigure[$\A_3$]{
\includegraphics[scale=.35]{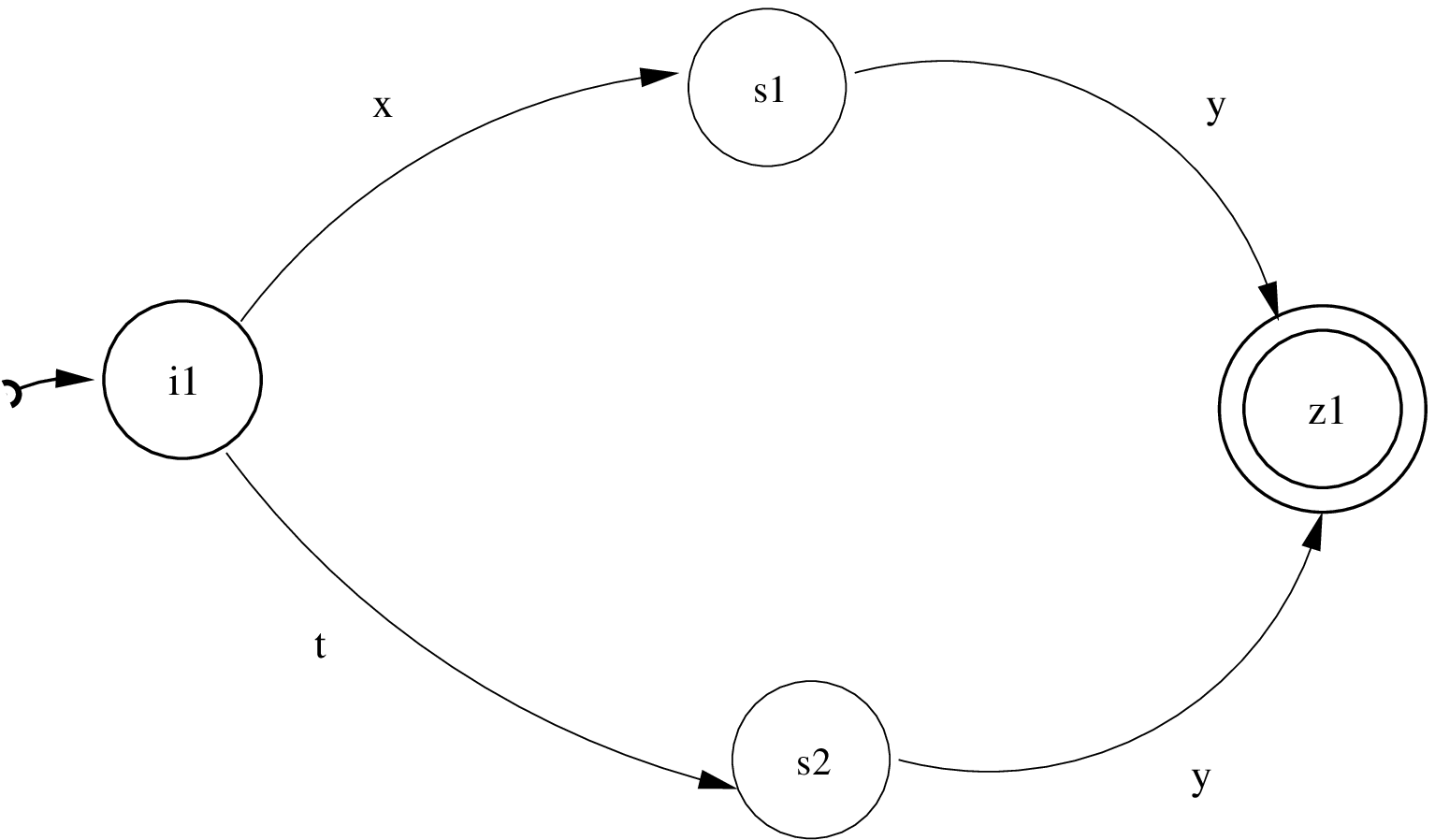}\label{figdeg3}}  \hspace{9mm}
\end{center}
\caption{The automata $\A_2, \A_3$.}\label{fig}
\end{figure}
%
%
\end{example}


\subsection{Thick semigroups and cones}\label{sub:thickthin}
We want to classify the subsets accepted by special automata over $F$ by modulo of their measure. This classification requires recalling of the notion of $X-$complete automaton, which was already used in  \cite{multiplicative} and \cite{fmrI} for analogous purposes.
Let $\A = (S, \S, \d, i_0, z_0)$ be a special automaton satisfying the conditions (a)--(f). $\A$ is called $\S-${\em complete} if for every state $s \in S \smallsetminus \{i_0\}$ of type $x \in \S$ every label from $\S \smallsetminus \{x^{-1}\}$ is present on one of the arrows exiting from $s$ and exactly $|\S|$ arrows exits from  $i_0$. 
Further, let $R_2$ be a regular set of the second type and $\A_2 = (S, \S, \d, z_0, z_0)$ be the corresponding automaton.
$\A_2$ is called $\S-${\em complete} if for every state $s \in S$ of type $x$ all arrows labeled by $\S \smallsetminus \{x^{-1}\}$ exit from $s$.
Otherwise $\A$ (or $\A_2$) is {\em not $\S-$complete} (for instance, the automaton $\A$ in Example, as well as $\A_1, \A_2$, and $\A_3$ are not $\S-$complete). The following proposition shows that the $\lambda-$ measure can be easily estimated in the latter case.
\begin{proposition}\label{notxcomplete} Let $\A$ be a special automaton satisfying the conditions (a)--(f), and $R=L(\A)$ be a saturated set such that $R = R_1 \circ R_2$ is the splitting of the form (\ref{RR1R2}), with $R_1 = L(\A_1)$ and $R_2=L(\A_2)$. If $\A_2$ is not $\S$-complete, then $R$ is exponentially $\lambda-$ measurable.
\end{proposition}
\begin{proof} If $\A$ is $\S-$complete, then only one of $\A_1, \A_2$ can be $\S-$complete since $R_1 \circ R_2$ is unambiguos by Lemma \ref{lenondegen}. Moreover, if $\A_2$ is $\S-$complete for the corresponding $\S-$complete automaton $\A$, then $\A_1$ is not $\S-$complete at the state $z_0$. Thus, we can assume that precisely $\A_2$ is not $\S-$complete. Then $R = L(\A)$ is $\lambda-$measurable by Theorem 3.4 in \cite{multiplicative}. Since $R$ is regular, it is exponentially $\lambda-$measurable by asymptotic classification of regular sets (\cite{multiplicative,af_semr}).
\end{proof}
 If, on the other hand, $\A$ is $\S-$complete, then we can improve some previously known results on  classification of regular subsets in free groups.
Namely, let $R_2$ be a regular subset of $F$ of second type accepted by the automaton $\A_2$. According to Corollary \ref{comonoid}, $R_2= L(\A_2)$ forms a (special) monoid, and if $\A_2$ is $\S-$complete, we shall call $R_2$ {\em thick}. An interesting fact about thick monoids is that we can describe them in terms of double-based cones.
We recall that the {\em cone} $C(w)$ is the set of all elements in $F$ containing $w$ as initial subword. In what follows we also shall be interested in a symmetric notion of a cone with a right-hand side handle, i.e. the set of all words in $F$ that terminates with $w$ (we denote this sort of cones by $C[w]$). Another member of this family is the {\em double-based cone} with (nontrivial) handles $w_1, w_2$, consisting of all words in $F$ of the form $w_1 \circ f \circ w_2$, $f \in F$. Notice that all three types of cones are regular in $F$ (see, for example, Corollary 3.15 \cite{dc}). Let us consider the {\em generalized $x$-cone} $C(Y,x)$, $x \in \S$, i.e. the union of double-based cones of the form $\mathop{\sqcup}\limits_{y \in \S: y \neq x^{-1}}C(y,x)$.

The following technical observation regarding generalized cones  give us first examples of thick monoids:

\begin{lemma}\label{le:thmonlang} Let $C(Y,x)$ be the generalized $x-$cone, $x \in \S$ and $M = C(Y,x)  \cup \{1\}$. Then
\begin{itemize}
\item[1.] $C(Y,x) =  C[x] \smallsetminus C(x^{-1},x)$, and $M$ is a thick monoid;
\item[2.] $g_M(t) = \fracd{(2m-1) t^2}{4m^2 (1-t)} + 1+ \fracd{t}{2m} + \fracd{t^2}{4m^2} + \fracd{t^3}{2m (2m - 1 -t)}$, and
\item[3.] $\mu_0(M) = \fracd{2m-1}{4m^2}$.
\end{itemize}
\end{lemma}

\begin{proof} The proof of 1. follows from the definitions of cones and thick monoids, while Example \ref{ex} (3) and Lemma \ref{nonio} [3.], [4.] below provide estimates on the generating function and the Cesaro density given in [2.] and [3.].
\end{proof}

Now we are ready to refine the asymptotic classification of regular sets in $F$ (see Introduction and Theorem 3.4 \cite{multiplicative} for comparison).
\begin{theorem}\label{thick} A regular subset $R$ of $F$ is thick if and only if it contains a subset $w\circ T$, with $T$ being a thick monoid and $w \in F$.
\end{theorem}

\begin{proof} Clearly, every set of the form $w \circ T$ is regular and thick (where $1 \circ T$, by definition, stands for $T$).
Suppose now $R$ is regular and thick. We decompose $R$ into a finite number of subsets as in Lemma \ref{Decomp1}. Since a finite union of exponentially $\lambda-$measurable subsets is exponentially $\lambda-$ measurable (see, for example, Proposition 4.1 \cite{fmrI}), without loss of generality one can suppose that $R$ is accepted by a special automaton or forms a special monoid. In the latter case, $R$ is a thick monoid itself, so suppose $R$ is a set accepted by a special automaton. In this case we apply Lemma \ref{lenondegen} to procure a pair of sets $R_1$ and $R_2$ of corresponding types, with $R_2$ being $\S-$complete by Proposition \ref{notxcomplete}. 
Since $R = R_1 \circ R_2$, the set $R$ contains a subset $w \circ R_2$, with $w \in R_1$. This completes the proof.
\end{proof}

\subsection{Computing $\lambda-$measure of regular sets}\label{sub:lmeasure}

Another immediate consequence of Lemma \ref{lenondegen} and Proposition \ref{notxcomplete} is an algorithm for computation of $\lambda-$measure of exponentially negligible regular set $R$ accepted by a special automaton $\A$. We assume that our reader is familiar with the concept of discrete-time Markov chain and refer to \cite{kesn} as one of the fundamental manuals on this subject.

 Let $\A = (S, \S, \d, i_0, z_0) $ be a special automaton over $F$ and let $R = R(\A)$ be a $\lambda-$measurable regular set. We split $\A$ into $\A_1$, $\A_2$ and $\A_3$, obtaining regular sets $R_1$, $R_2$ and $R_3$ (without loss of generality, one can consider the case when all these sets are non-empty). Further, due to formula (\ref{gen}) and Proposition \ref{notxcomplete}, it is enough to calculate the value of $\lambda-$ measure for $R_3$, accepted by the special automaton $\A_3 = (S_3, \S, \d_3, i_1, z_1)$.

%
Consider a finite Markov chain $\mathcal{M}$ with the same states as in $\A_3$ together with an additional dead state $D$. We set transition probabilities
from $z_1$ to $z_1$ and from $D$ to $D$ being equal $1$. Every arrow from a state $s$ in $\A_3$ gives the corresponding transition from the state $s$ in $\mathcal{M}$ which we assign  the transition probability $\fracd{1}{2m-1}$. If at some state $s$ of type $x$ in $\A_3$ there is no exiting arrow labeled $y \in
\S \smallsetminus \{x^{-1}\}$, we make a transition from $s$ to $D$ in $\mathcal{M}$ assigning it the probability $\fracd{1}{2m-1}$. We take the stochastic vector being zero everywhere except the state $i_1$ (so it have the only nontrivial entry $1$ at the state $i_1$). This complete the description of the Markov chain $\mathcal{M}$. Clearly, the states $z_1$ and $D$ of Markov chain $\mathcal{M}$ are absorbing, and all other states are transient. Obviously, $P(z_1) = \lambda(R_3)$, and it was shown in \cite{multiplicative,fmrI} that $\lambda(R_3) < 1$ for any $\lambda-$measurable set $R$. Therefore, one can calculate $\lambda(R_2)$ using formula (\ref{R2R3}). A similar argument allows to compute $\lambda(R_1)$, and so we are done.
Thus, the Markov chain $\mathcal{M}$ provides us with the following algorithm for computation of $\lambda(R)$.

{\bf{Algorithm II: }}
Let $\A$ be a special automaton and $R = R(\A)$ be $\lambda-$measurable.
\begin{itemize}
\item[1.] Split $\A$ into $\A_1, \A_2, \A_3$ as in Lemma \ref{lenondegen}.
\item[2.] Construct Markov chains for $\A_1 = (S_1, \S, \d_1, i_0, z_0)$ and $\A_3 = (S_3, \S, \d_3, i_1, z_1)$.
\item[3.] Calculate the probabilities $P(z_0)$ and $P(z_1)$; so $\lambda(R_1) = P(z_0)$ and $\lambda(R_3) = P(z_1)$.
\item[4.] Compute $\lambda(R_2) = \mathop{\sum}\limits_{i=0}^{\infty}P^i(z_1) < \infty$.
\item[5.] Finally, compute $\lambda(R) = \lambda(R_1) \cdot \lambda^{\ast}(R_2)$.
\end{itemize}

%


\section{Computations}\label{sec:comp}
In this section we carry out all necessary measurements of double-based cones and thick monoids.

\subsection{Generating functions and Cesaro density of double-based cones}\label{sub:doublecones}
This technical but crucial lemma will supply us with data about generating function and values of Cesaro density for double-based cones.

\begin{lemma}\label{nonio} Let $C(a,b)$ be a double-based cone with both handles $a,b$ in $\S$. Then following holds:
\begin{itemize}
\item[1.] $f_k(C(a, b)) = f_k(C(c, d))$ and therefore $g_{C(a,b)}(t) = g_{C(c,d)}(t)$ for all $a, b, c, d$ in $\S$ such that
$ab \neq 1$, $cd \neq 1$. Further, $f_k(C(a, a^{-1})) = f_k(C(b, b^{-1}))$ for arbitrary $a, b \in \S$.

\item[2.] $f_k(C(a,a^{-1})) = (2m-1)f_{k}(C(a,a))-\fracd{1}{2m(2m-1)^{k-1}}$, for $k \ge 3$,

\item[3.] $ g_{C(a,a)}(t) = \fracd{t^2}{4m^2 (1-t)} + \fracd{t^2}{4m^2(2m-1)} + \fracd{t^3}{2m (2m-1) (2m - 1 -t)}$, and
$g_{C(a,a^{-1})}(t) = \fracd{t^2}{4m^2 (1-t)} - \fracd{t^2}{4m^2} - \fracd{t^3}{2m (2m - 1 -t)}$,

\item[4.] $\mu_0(C(a,b)) = \mu_0(C(c,d)) = \fracd{1}{4m^2}$ for all $a,b,c,d \in \S$.

\end{itemize}
\end{lemma}

\begin{proof}  Notice first, that $n_k(C(a,a^{-1})) = n_k(C(b,b^{-1}))$ (recall that $n_k(L) =  |L \cap S_k|$, i.e. the number of elements of length $k$ in $L$), for all $a, b \in \S$. The same equalities holds between the other double-based cones: $n_k(C(a,b)) = n_k(C(c,d))$ for all $a,b,c,d \in \S$ such that $a b \neq 1$ and $c d \neq 1$. This proves the first claim.

To prove 2., 3., and 4. we are going to construct a bijective map $\psi: C(a,a) \rightarrow C(a,a^{-1}) \sqcup \{a^n\}_{n>0}$.
For every element $u \in C(a,a)$ of the form $u = a^l \circ f_0 \circ a^m$, where $l$ and $m$ maximal, i.e. $f_0 \neq 1$ does not starts with $a$ and does not end with $a$, define $\psi(u) = a^l \circ f_0 \circ a^{-m}$. If, on the hand, $u \in C(a,a)$ has a form $u = a^l$, then take $\psi(a^l)= a^l$. Clearly, $\psi$ is bijective and therefore $f_k(C(a,a)) = f_k(C(a,a^{-1}) + f_k(\{a^n\}_{n>0})$ for $k >2$.
Since $C(a) = \mathop{\sqcup}\limits_{b \in \S} C(a,b)$, and due to the equality $f_k(C(a)) = \fracd{1}{2m}$, we have
$$\fracd{1}{2m} = (2m-1)f_k(C(a,a)) + f_k(C(a,a^{-1})) {\textrm{ for }} k >2.$$ $${\textrm{But }}f_k(C(a,a^{-1})) = f_k(C(a,a)) - \fracd{1}{2m(2m-1)^{k-1}}, {\textrm{ and therefore }}$$
$$\fracd{1}{2m} = 2m f_k(C(a,a)) - \fracd{1}{2m(2m-1)^{k-1}} {\textrm{ for }} k>2.$$ To compute generating functions of corresponding sets, we multiply $f_k$ with $t^k$ and take an infinite sum of these products. As a result we obtain:
$$\mathop{\sum}\limits_{k=2}^{\infty} f_k (C(a))t^k = (2m-1) \mathop{\sum}\limits_{k=2}^{\infty} f_k (C(a,a))t^k + \mathop{\sum}\limits_{k=3}^{\infty} f_k (C(a,a^{-1}))t^k$$ and therefore
$$\fracd{1}{2m} \mathop{\sum}\limits_{k=2}^{\infty} t^k = (2m-1) f_2(C(a,a))t^2 + 2m \mathop{\sum}\limits_{k=3}^{\infty} f_k (C(a,a))t^k,$$
from which follows
$$\fracd{t^2}{2m (1-t)} = 2m g_{C(a,a)}(t) - \fracd{t^2}{2m(2m-1)} - \mathop{\sum}\limits_{k=3}^{\infty} \fracd{t^k}{2m(2m-1)^{k-1}}.$$
Hence,
$$ g_{C(a,a)}(t) = \fracd{t^2}{4m^2 (1-t)} + \fracd{t^2}{4m^2(2m-1)} + \fracd{t^3}{2m(2m-1) (2m - 1 -t)}, $$
and therefore
$$ g_{C(a,a^{-1})}(t) = \fracd{t^2}{4m^2 (1-t)} - \fracd{t^2}{4m^2} - \fracd{t^3}{2m (2m - 1 -t)}.$$

Applying Corollary \ref{observ_g}, from the last two equalities we deduce
$$\mu_0(C(a,a)) =  - \mathop{\lim}\limits_{t \rightarrow 1} (t-1) g_{C(a,a)}(t) = \fracd{1}{4 m^2}$$
as well as  $\mu_0(C(a,a^{-1}))  = \fracd{1}{4 m^2}$. \end{proof}

Lemma \ref{nonio} can be easily generalized to the case of an arbitrary double-based cone in $F$.
\begin{theorem}\label{th:dbcone} Let $R = C(u,v)$ be a double-based cone with handles $u,v$ in $F$ such that $u = u_0 \circ a$, $v = b \circ v_0$, where $u_0, v_0 \in F$ and $a, b \in \Sigma$. Then
\begin{itemize}
\item[1.] $g_R(t) = g_{C(a,b)}(t) \cdot \lambda^*(u_0) \cdot \lambda^*(v_0)$;
\item[2.] $\mu_0(R) = \fracd{\lambda^*(u_0) \cdot \lambda^*(v_0)}{4m^2}$.
\end{itemize}
\end{theorem}
\begin{proof} The proof follows immediately from Lemma \ref{nonio} and definitions of generating function and $\lambda-$measure.
\end{proof}

It remains to show how one can compute both generating function and Cesaro measure of a thick monoid.

\begin{theorem}\label{Leo} Let $\A$ be a special automaton over $F$ such that $\A_2$ and $\A_3$ are automata from the decomposition (\ref{RR1R2}), (\ref{R2R3}) of Lemma \ref{lenondegen}. Suppose $T = L(\A_2)$ is a thick monoid and $z_3$ is of type $x$. Then
\begin{itemize}
\item[1)] $\mathop{\sqcup}\limits_{i=1}^k T \cdot w_i = \mathop{\sqcup}\limits_{j=1}^l C(A,x) \cdot v_j$ with $C(A,x)$ being generalized $x-$cone, $w_i, v_j \in F$ and $k,l < \infty$;
\item[2)] $g_T(t)$ and $\mu_0(T)$ can be computed effectively by $\A_3$.
\end{itemize}
\end{theorem}
\begin{proof} Let $\overline{T}$ be a prefix closure of $T$. Then
\begin{equation}\label{w}
\overline{T} = \mathop{\sqcup}\limits_{w_i \in W} T \cdot w_i
\end{equation}
for a set $W$ such that $W = \{w_i \in F:$ there is a simple path $p_i$ in $\A_3$ such that $p_i$ starts at $s \in S$, ends at $z_3$ and $w_i^{-1}$ is a label of $p_i \}$. Clearly, $W$ is finite.
On the other hand,
\begin{equation}\label{v}
\overline{T} = \mathop{\sqcup}\limits_{v_j \in V} C(A,x) \cdot v_j
\end{equation}
for some finite set $V$ of words in $F$, defined by $\A_3$.
Then claim 1) follows from (\ref{w}) and (\ref{v}), while claim 2) follows from 1) and Lemma \ref{le:thmonlang}.
\end{proof}


\begin{thebibliography}{0}
\bibitem{af_semr} Ya.~S.~Averina and E.~V.~Frenkel, On strictly sparse subsets of a free group, (Russian) {\it Siberian Electronic Mathematical Reports } (2005), vol. 2 pp. 1 -- 13, http://semr.math.nsc.ru

\bibitem{multiplicative}  A.~V.~Borovik, A.~G.~Myasnikov and V.~N.~Remeslennikov, Multiplicative measures on free groups, {\it Intern. J. of Algebra
and Computation}, 13 (2003), {\bf 6}, pp. 705 -- 731.


%
%
\bibitem{DMR}  E.~Yu.~Daniyarova, A. G. Myasnikov, V. N. Remeslennikov, Dimension in universal algebraic geometry, {\it Doklady of Academy of Science}, {\bf 457} no. 3 (2014) pp.~265 -- 267
  
\bibitem{eps} D.~Epstein, J.~Cannon, D.~Holt, S.~Levy, M.~Paterson and W.~Thurston, {\it Word Processing in Groups} (Jones and Bartlett, Boston, 1992).


\bibitem{flajolet} P.~Flajolet and R.~Sedgwick, ``Analytic Combinatorics: Functional Equations, Rational and Algebraic Functions'', Res.
Rep. INRIA RR4103, January 2001, 98~pp.

\bibitem{fmrI} E.~Frenkel, A.~G.~Myasnikov and V.~N.~Remeslennikov,
Regular sets and counting in free groups, in {\it Combinatorial and
Geometric Group Theory}, Series ``Trends in Mathematics'', (Birkhauser Verlag Basel/Switzerland, 2010),
pp.~93--118.


\bibitem{fmrII} E.~Frenkel, A.~G.~Myasnikov and V.~N.~Remeslennikov, Amalgamated products of groups: measures of random normal
forms, {\it Fund. Appl. Math.} {\bf 16}(8) (2010), pp.~189-221.


\bibitem{dc} E.~Frenkel, V.~N.~Remeslennikov, Double cosets in free groups, {\it Intern. J. of Algebra
and Computation} {\bf 23} (5) (2013), pp.~1225 -- 1241.

\bibitem{almaz} E.~Frenkel, V.~N.~Remeslennikov, Cones and thick monoids in free groups, {\it Materials of International Workshop 'Almaz-2'}, pp.64--68, Omsk, 2015.

\bibitem{Gil} R.~Gilman, Formal languages and their application to combinatorial group theory in Groups, {\it Languages Algorithms}, Contemp. Math., 378, Amer. Math. Soc., 2005, pp.~1-36



%
\bibitem{km} I.~Kapovich and A.~G.~Myasnikov, Stallings foldings and
subgroups of free groups, {\it J. Algebra} {\bf 248} (2002), pp.~608 -- 668.


\bibitem{kesn} J.~G.~Kemeny, J.~L.~Snell, {\it Finite Markov chains}, Princeton, NJ: Van Nostrand, 1960

\bibitem{lawson} M.~V.~Lawson, ``Finite automata. {\it Chapman} \& Hall/CRC, 2004.

\bibitem{RS} M.~O.~Rabin and D.~Scott, Finite automata and their decision problems, {\it IBM Journal of Research and Development} {\bf 3 (2)} (1959), pp.~114 -- 125.

\end{thebibliography}
\end{document}